\documentclass[12pt,psamsfonts,oneside,reqno]{amsart}

\usepackage{thmtools}
\usepackage{amsthm}

\usepackage{graphicx}
\usepackage{caption}
\usepackage{subcaption}

\usepackage{cite}

\usepackage[margin=1.2in]{geometry}

\usepackage{amsmath, amssymb}

\usepackage{hyperref}
    \hypersetup{colorlinks, breaklinks,
                linkcolor = blue,
                urlcolor = blue,
                anchorcolor = blue,
                citecolor = blue}

\usepackage{cleveref}

\usepackage{multicol}

\numberwithin{equation}{section}
\numberwithin{figure}{section}

\newcommand{\cC}{\mathcal{C}}
\newcommand{\cE}{\mathcal{E}}
\newcommand{\bbR}{\mathbb{R}}
\newcommand{\bbZ}{\mathbb{Z}}

\newcommand{\sidelengthvariable}{s}
\newcommand{\multia}{\alpha}
\newcommand{\multib}{\beta}
\newcommand{\bigF}{F}

\newcommand{\<}{\left\langle}
\renewcommand{\>}{\right\rangle}

\DeclareMathOperator{\dist}{dist}
\DeclareMathOperator{\diam}{diam}

\declaretheorem[numberwithin=section,name=Theorem]{thm}
\declaretheorem[sibling=thm,name=Lemma]{lem}
\declaretheorem[sibling=thm,name=Corollary]{cor}

\declaretheorem[sibling=thm,name=Definition,style=definition,
qed=$\Diamond$]{defn}

\declaretheorem[sibling=thm,name=Remark,style=remark,
qed=$\Diamond$]{rem}

\title{The Whitney extension theorem in high dimensions}% (\jobname)}
\author[Alan Chang]{Alan Chang}\email{\textcolor{blue}{\href{mailto:ac@math.uchicago.edu}{ac@math.uchicago.edu}}}

\address{Dept. of Mathematics, University of Chicago, 5734 S. University Avenue, Room 208C, Chicago, Illinois 60637}

\subjclass[2010]{58C25}

\keywords{Whitney extension theorem.}

\date{July 22, 2015}

\begin{document}

\begin{abstract}
We prove a variant of the standard Whitney extension theorem for $\cC^m(\bbR^n)$, in which the norm of the extension operator has polynomial growth in $n$ for fixed $m$.
\end{abstract}

\maketitle

%\tableofcontents

\section{Introduction}

Let $E \subset \bbR^n$ be a compact set. Let $F : \bbR^n \to \bbR$ be a $\cC^m$ function. For each point $y \in E$, define 
\begin{align}
\label{eq:taylor-polynomials-of-f}
P_{y}(x) 
= 
\sum_{|\multia| \leq m} \partial^{\multia}\bigF(y)\frac{(x-y)^\multia}{\multia!}.
\end{align}
Thus, we obtain a \emph{Whitney field}, that is, a family $(P_y)_{y \in E}$ of polynomials of degree $\leq m$ indexed by the points of $E$. We write $F|_E = (P_y)_{y \in E}$ to indicate that the Whitney field $(P_y)_{y \in E}$ arises from the function $F$.

We would like to know which Whitney fields arise from a $\cC^m(\bbR^n)$ function in this way. Taylor's theorem gives us some necessary conditions, namely:
\begin{align}
\label{eq:defn-C-m-E-condition-1}
&\sup_{\substack{x \in E\\ |\alpha| \leq m}}
|\partial^\multia P_{x}(x)|
\ \text{ and } \ 
\sup_{\substack{x,y \in E\\x \neq y\\ |\alpha| \leq m}}
\frac{|\partial^\multia (P_{x}-P_y)(x)|} {|x-y|^{m-|\multia|}}
\ \text{ are both finite.}
\\
\label{eq:defn-C-m-E-lim-0}
&\lim_{|x-y| \to 0}
\frac{|\partial^\multia (P_{x} - P_{y})(x)|}{|x-y|^{m-|\multia|}} 
=
0
\ \text{ for all } |\alpha| \leq m.
\end{align}
The Whitney extension theorem asserts that these necessary conditions are in fact sufficient \cite{whitney}. We can view the extension from Whitney fields to functions $\bbR^n \to \bbR$ as a linear map between the normed spaces $\cC^m(E)$ and $\cC^m(\bbR^n)$, which we define now.

\begin{defn}%[$\cC^{m}(\bbR^n)$]
Let $m \in \bbZ_{\geq 0}$. Suppose that $\bigF : \bbR^n \to \bbR$ is $m$-times continuously differentiable. Then we say $\bigF \in \cC^m(\bbR^n)$ if $M := \sup \{|\partial^\multia  \bigF(x)| : x \in \bbR^n, |\alpha| \leq m \} < \infty$. The $\cC^m(\bbR^n)$ norm of $\bigF$ is $M$. 
\end{defn}

\begin{defn}%[$\cC^{m}(E)$]
\label{defn:C-m-omega-E}
An element of $\cC^{m}(E)$ is a Whitney field $(P_y)_{y \in E}$ such that \eqref{eq:defn-C-m-E-condition-1}, and \eqref{eq:defn-C-m-E-lim-0} hold. The $\cC^{m}(E)$ norm of $(P_y)_{y \in E}$ is the greater supremum in \eqref{eq:defn-C-m-E-condition-1}. 

Note: We will write $(P_y)$ as shorthand for $(P_y)_{y \in E}$. We will also write $f$ to denote a Whitney field $(P_y)$.
\end{defn}

In this paper, we prove the following.

\begin{thm}
[Whitney extension theorem, polynomial bound on norms]
\label{thm:intro-whitney-poly-bound}
There exists a linear operator $\cC^{m}(E) \to \cC^{m}(\bbR^n)$ such that if $f \in \cC^m(E)$ is mapped to $\bigF \in \cC^m(\bbR^n)$, then
\begin{align}
\label{eq:defn-extension-operator}
\text{$F|_{E} = f$ and  
$F$ has derivatives of all orders on $E^c$.}
\end{align}
Furthermore, the norm of the operator is at most $Cn^{5m/2}$, where $C$ depends only on $m$.
\end{thm}

A linear operator $\cE_m$ satisfying \eqref{eq:defn-extension-operator} is an \emph{extension operator}. In the well-known proof of the Whitney extension theorem, the bound on the norm of the operator grows exponentially with the dimension $n$. Consequentially, our main result in this paper is the construction an extension operator $\cC^m(E) \to \cC^m(\bbR^n)$ that yields a bound having polynomial growth in $n$ for fixed $m$. 

In the standard proof, the exponential growth of the the dimension is due to the behavior of ``cutoff functions'' (defined in \Cref{sec:cutoff-functions}) of ``Whitney cubes'' (defined in \Cref{sec:std-w-e-t}) near their boundaries. We take care of this by averaging over translations of the cubes, which eliminates the problematic behavior.

\begin{rem}
Whitney \cite{whitney} gave the original proof of the theorem. Glaeser \cite{glaeser} presented a version of the Whitney extension theorem using the space $\cC^{m,\omega}(E)$, where $\omega$ is a modulus of continuity.\footnote{The space $\cC^{m,\omega}(E)$ is referred to as $\operatorname{Lip}(m+\omega, E)$ in \cite[Section 4, (4.6)]{stein1970}.} Although we will work mainly with $\cC^m(E)$, we will invoke Glaeser's result in the proof of \Cref{lem:avg-ext-maps-Cm-omega} to show that the average over a particular family of functions satisfies certain regularity conditions.
 
In this paper, we often refer to the proof given in Stein's textbook \cite{stein1970}. Stein presents a proof of Glaeser's version of the theorem; to obtain a proof of the $\cC^m(E)$ version, we simply drop all instances of $\omega$ from the proof.
\end{rem}

\begin{rem}
For $\cC^m(\bbR^n)$ and $\cC^m(E)$, there are many equivalent choices of norms. For example, we could also define the $\cC^m(\bbR^n)$ norm of $F$ via
$
\sup_{x \in E} (
\sum_{|\alpha| \leq m}
|
\partial^\alpha F(x)
|^2
)^{1/2}
$. However, the choice of norm does not affect our main result. (The exponent $5m/2$ in \Cref{thm:intro-whitney-poly-bound} might change, but the bound will remain a polynomial in $n$.)
\end{rem}

\begin{rem}
On $\cC^{1, 1}(\bbR^n)$, the space of functions with a Lipschitz derivative, we could define a norm via the Lipschitz constant of the derivative:
\begin{equation}
\|F\|_{\cC^{1,1}(\bbR^n)}
=
\sup_{\substack{x, y \in \bbR^n \\ x \neq y}}
\frac{|\nabla F(x) - \nabla F(y)|}{|x-y|}
.
\end{equation}
Given a Whitney field $(P_y) \in \cC^{1,1}(E)$, the papers of Le Gruyer and Wells calculate exactly the least possible $\cC^{1, 1}(\bbR^n)$ norm for any extension of $(P_y)$ and provide a construction that attains this minimum \cite{legruyer, wells}. This is a remarkable result, but the method does not seem to apply to other cases.
\end{rem}

\begin{rem}
Unlike Le Gruyer, we are not finding an extension with minimal norm. Instead, we are computing the least possible norm up to a constant that depends only on the dimension $n$. 

That is, suppose our extension operator (from \Cref{thm:intro-whitney-poly-bound}) maps $(P_y) \in \cC^m(E)$ to $F \in \cC^m(\bbR^n)$, while $\tilde F \in \cC^m(\bbR^n)$ is some competing extension. The composition of restriction followed by extension gives us $\tilde F \mapsto (P_y) \mapsto F$. It is a simple consequence of Taylor's theorem that the norm of the restriction operator is $O(n^m)$. (See \Cref{sec:restriction-norm}.) Combining this result with \Cref{thm:intro-whitney-poly-bound} guarantees that $\| F \|_{\cC^m(\bbR^n)} \leq C n^{7m/2} \| \tilde F \|_{\cC^m(\bbR^n)}$, where $C$ only depends on $m$.  Thus, $F$ has the least possible $\cC^m$ norm up to a factor of $Cn^{7m/2}$.
\end{rem}

\section{Cutoff functions}\label{sec:cutoff-functions}

The proof of the Whitney extension theorem relies on \emph{cutoff functions}. (We say $\phi : \bbR^n \to \bbR$ is a \emph{cutoff function} of the cube $Q \subset \bbR^n$ if $\phi$ is $1$ on $Q$ and $0$ outside of $Q^*$, where $Q^*$ is a cube with the same center as $Q$ but with a larger side length.)

\subsection{Construction of cutoff functions}

To construct cutoff functions, we start by fixing a $\sigma : \bbR \to \bbR$  with the following properties: 
(1)
$\sigma(x) = 0$ if $x \leq -1$,
(2) 
$\sigma(x) = 1$ if $x \geq 0$,
(3)
$\sigma$ is nondecreasing,
(4)
$\sigma \in C^\infty$.

Next, let $t \in (0,\frac{1}{4})$. (Eventually, we will choose $t = 1/n$, but for most of the paper, $t$ will be a parameter independent of $n$.) Consider the function $\vartheta : \bbR \to \bbR$, where $\vartheta(x)$ is given by $0, \sigma(\frac{x}{t}), 1, \sigma(\frac{1-x}{t}), 0$ on $(-\infty, -t]$, $[-t, 0]$, $[0, 1]$, $[1, 1+t]$, $[1+t, \infty)$, respectively. Then $\vartheta$ is a smooth \emph{cutoff function} of the unit interval $[0,1]$. 

For $x = (x_1, \ldots, x_n) \in \bbR^n$, we define 
$
\Theta(x) = \vartheta(x_1) \cdots \vartheta(x_n).
$
Then $\Theta(x)$ is a smooth cutoff function for the unit $n$-cube $[0,1]^n$. For $b = (b_1, \ldots, b_n) \in \bbR^n$ and $\sidelengthvariable > 0$, the cube $Q = [b_1, b_1 + \sidelengthvariable] \times \cdots \times [b_n, b_n + \sidelengthvariable]$ will have cutoff function
\begin{align}\label{eq:cutoff-function-using-Theta}
\phi_Q(x)
=
\Theta\left(\frac{x-b}{\sidelengthvariable}\right)
.
\end{align}
Note that $\vartheta$ and $\Theta$ both depend on $t$, but we do not write this dependence explicitly. The parameter $t$ affects how quickly the cutoff functions ``drop from $1$ to $0$.'' 

\subsection{Estimates on sums over cutoff functions}\label{sec:estimates-on-tilings}

Consider a tiling of $\bbR^n$ by cubes of unit length. We would like to obtain estimates on sums over cutoff functions of these cubes (since such estimates would also give us estimates on sums over a subset of the terms). 

Let $\psi : \bbR \to \bbR$ be given by $\psi(x) = 1$ if $\dist(x, \bbZ) \leq t$ and $\psi(x) = 0$ otherwise. (In other words, $\psi$ is the characteristic function of ``numbers close to the integers.'')

\begin{lem}%[$n$-dimensional estimate]
\label{lem:n-dim-estimate-Theta}
For all $x = (x_1, \ldots, x_n) \in \bbR^n$ and multi-indices $\multia$,
\begin{align}
\label{eq:n-dim-estimate-Theta}
\sum_{a \in \bbZ^n} \left|\partial^\multia \Theta(x-a)\right|
\lesssim
t^{-|\multia|} \prod_{i=1}^n (\psi(x_i) + 1)
.
\end{align}
The implied constant does not depend on $n$ or $t$.
\end{lem}

\begin{proof}
Because $\Theta(x) = \vartheta(x_1) \cdots \vartheta(x_n)$, it suffices to prove the one-dimensional version of \eqref{eq:n-dim-estimate-Theta}. That is, we need to show for all $x \in \bbR$ and $j \in \bbZ_{\geq 0}$,
\begin{align}
\label{eq:1-dim-estimate-Theta}
\sum_{k \in \bbZ} \left|\vartheta^{(j)}(x - k)\right|
\lesssim
t^{-j} (\psi(x) + 1)
.
\end{align}

If $\dist(x, \bbZ) > t$, then $x \in (a + t, a+1-t)$ for some $a \in \bbZ$ and we have 
$
\sum_{k \in \bbZ} |\vartheta^{(j)}(x - k)| 
= 
|\vartheta^{(j)}(x - a)|
\leq 1
$
(since $\vartheta$ is identically $1$ in a neighborhood of $x-a$).

If $\dist(x, \bbZ) \leq t$, then $x \in (a - t, a+t)$ for some $a \in \bbZ$ and we have 
$
\sum_{k \in \bbZ} |\vartheta^{(j)}(x - k)| = |\vartheta^{(j)}(x-a)| + |\vartheta^{(j)}(x-(a-1))| \leq c_j t^{-j}
$
with $c_j = 2 \sup_{x \in \bbR} |\sigma^{(j)}(x)|$.

Putting the two cases together gives us \eqref{eq:1-dim-estimate-Theta}.
\end{proof}

\begin{rem}\label{rem:implied-constants}
For the rest of this paper, unless otherwise noted, the implied constant in an expression $A \lesssim B$ can depend only on $\sigma$, $m$ and $|\alpha|$.\footnote{It might seem unnecessary to state that the implied constant can depend on $|\alpha|$, since in most cases we work with $|\alpha| \leq m$. However, in certain situations, we work with derivatives $\alpha$ of all orders, including $|\alpha| > m$.} In particular, the constant cannot depend on $n$ or $t$.
\end{rem}

\section{The Whitney cube partition and the extension operator $\cE_m$}\label{sec:std-w-e-t}

\subsection{Notation and basic properties}

We will write the \emph{Whitney cube partition} of $E^c$ as $E^c = \bigcup_{k=0}^\infty Q_k$, where each $Q_k \subseteq \bbR^n$ is a closed cube with sides parallel to the axes and for $j \neq k$, the interiors of the cubes $Q_j$ and $Q_k$ are disjoint. The $Q_k$ are the \emph{Whitney cubes}. \cite[Section 1]{stein1970} gives a construction of the Whitney cube partition. For this paper, it is not important to know the details of the construction. We will rely mainly on the properties listed below in \Cref{lem:basic-whitney-cube-properties}.

We let $\phi_k(x)$ denote the cutoff function of the cube $Q_k$, given by \eqref{eq:cutoff-function-using-Theta}. Then we set $S(x) = \sum_{k} \phi_k(x)$ and $\phi_k^*(x) = \phi_k(x)/S(x)$. (Hence $\{\phi_k^*(x)\}_k$ is a \emph{partition of unity}.) We define the closed cube $Q_k^*$ as follows: if $Q_k$ is centered at $x$ and has side length $\sidelengthvariable$, then $Q_k^*$ is centered at $x$ and has side length $\sidelengthvariable(1+2t)$. Note that $Q_k^*$ contains the support of $\phi_k$ (and $\phi_k^*$).

We will need some auxiliary functions to be used in our estimates:

\begin{itemize}
\item
$\delta(x) = \dist(x, E)$.
\item
$\Psi : E^c \to \bbR$ is given by
\begin{align}
\label{eq:defn-psi-x}
\Psi(x)
=
\sum_{c_1' \delta(x) n^{-1/2} \ \leq \ p \  \leq \  c_2'\delta(x) n^{-1/2}}
\left[
\prod_{i=1}^n \left( 
\psi\left(\frac{x_i}{p}\right) + 1
\right)
\right]
\end{align}
where the sum over $p$ is over powers of $2$. Also, in \eqref{eq:defn-psi-x}, $c_1'$ and $c_2'$ are absolute constants that will be specified later, in \Cref{subsec:estimates-partition-unity}. (Recall that $\psi$ is defined in \Cref{sec:estimates-on-tilings}.)

\end{itemize}

Some basic properties of the Whitney cube partition include the following, which we repeat from \cite{stein1970}.

\begin{lem}[Basic properties of Whitney cubes]
\label{lem:basic-whitney-cube-properties}
\leavevmode %http://tex.stackexchange.com/questions/73741/enumerate-alignment-problem-in-theorem-environment
\begin{itemize}
\item
All Whitney cubes are dyadic. (We say a cube $Q \subset \bbR^n$ is \emph{dyadic} if its side length $\sidelengthvariable$ is a power of $2$ and its vertices all lie in the lattice $\sidelengthvariable \bbZ^n$.)
\item
If two Whitney cubes of side lengths $\sidelengthvariable_1$ and $\sidelengthvariable_2$ intersect at their boundaries, then 
$
\frac{1}{4}\sidelengthvariable_2 \leq \sidelengthvariable_1 \leq 4\sidelengthvariable_2. 
$
\item
If two distinct Whitney cubes, say $Q_1$ and $Q_2$, are such that $Q_1$ and $Q_2^*$ intersect, then $Q_1$ and $Q_2$ intersect at their boundaries.
\item
Let $x \in Q_k^*$. Let $s$ be the side length of $Q_k$. Then 
\begin{align}
\label{eq:delta-dist-comparable}
\delta(x) \approx \dist(Q_k, E) \approx \dist(Q_k^*, E) \approx \diam(Q_k) \approx \diam(Q_k^*) \approx sn^{1/2}
,
\end{align}
where $\diam(Q)$ denotes the diameter of the cube $Q$, and $A \approx B$ means both $A \lesssim B$ and $B \lesssim A$ hold. The implied constants are absolute.
\end{itemize}
\end{lem}

\subsection{Estimates on the partition of unity}
\label{subsec:estimates-partition-unity}

We need the following estimate on the partition of unity.

\begin{lem}
\label{lem:derivatives-of-sum-phi-star}
Let $x \in E^c$, let $Q$ be a Whitney cube that contains $x$, and let $\sidelengthvariable$ be the side length of $Q$. Then
\begin{align}
\sum_k
|\partial^\multia  \phi_k^*(x)|
&\lesssim
(\sidelengthvariable t)^{-|\multia|}
\Psi(x)^{|\multia|+1}.
\end{align}
Following our conventions on the use of $\lesssim$ (see \Cref{rem:implied-constants}.), the implied constant does not depend on $n$, $t$ or $\sidelengthvariable$. 
\end{lem}

\begin{proof}
Observe that $\partial^\multia  \phi_k(x) \neq 0$ implies $x \in Q_k^*$. Then by \Cref{lem:basic-whitney-cube-properties}, the side lengths of $Q$ and $Q_k$ differ by at most a factor of 4. Thus, we have %
\begin{align}
\label{eq:estimates-derivative-S}
|\partial^\multia  S(x)| 
\leq 
\sum_k |\partial^\multia  \phi_k(x)|
\leq
\sum_{\frac{\sidelengthvariable}{4} \leq p \leq 4\sidelengthvariable}
\sum_{a \in \bbZ^n} 
\left|
\left(\frac{\partial}{\partial x}\right)^\multia\Theta\left(\frac{x-a}{p}\right)
\right|
,
\end{align}
where $p$ ranges over powers of $2$. By \eqref{eq:delta-dist-comparable}, there exist absolute constants $c_1'$ and $c_2'$ such that $\frac{\sidelengthvariable}{4} \leq p \leq 4\sidelengthvariable$ implies $c_1' \delta(x) n^{-1/2} \leq p \leq c_2'\delta(x) n^{-1/2}$. Then the right-hand side of \eqref{eq:estimates-derivative-S} is $\lesssim (st)^{-|\multia|} \Psi(x)$ by \Cref{lem:n-dim-estimate-Theta}. This implies that $\left|
\left(\frac{\partial}{\partial x}\right)^\multia\left[ S(x)^{-1} \right]
\right|
\lesssim
(\sidelengthvariable t)^{-|\multia|}
\Psi(x)
^{|\multia|}$, and the lemma follows.
\end{proof}

\subsection{The extension operator $\cE_m$}

Since $E$ is closed, for each $k$, we can find a point $p_k \in E$ such that $\dist(Q_k, E) = \dist(Q_k, p_k)$. Then for $f=(P_y) \in \cC^{m,\omega}(E)$, we set
\begin{align}
\label{eq:def-extm}
\cE_m(f)(x)
=
\begin{cases}
P_x(x) 
&\qquad\text{if } x \in E
\\
\sum_k' P_{p_k}(x) \phi_k^*(x)
&\qquad\text{if } x \in E^c.
\end{cases}
\end{align}
where the $\sum'$ means to sum over cubes $Q_k$ such that $\dist(Q_k, E) \leq 1$. The Whitney extension theorem asserts that $\cE_m$ is a linear extension operator $\cC^m(E) \to \cC^m(\bbR^n)$. To determine the growth of the norm of $\cE_m$ when $n$ increases, we need more precise estimates.

\begin{lem}
\label{lem:derivatives-of-extm}
Suppose $f = (P_y) \in \cC^{m}(E)$ with norm $\leq 1$. Let $\bigF = \cE_m(f)$. Let $x \in E^c$ and $a \in E$. Then
\begin{align}
\label{eq:stein-a}
|\bigF(x) - P_{a}(x)| 
&\lesssim 
n^m |x - a|^m
.
\\
\label{eq:stein-a-prime}
|\partial^{\multia}\bigF(x) - \partial^{\multia}P_{a}(x)| 
&\lesssim
n^{m+|\multia|/2} t^{-|\multia|} \Psi(x)^{|\multia|+1} |x-a|^{m-|\multia|}
&&\text{for } |\multia| \leq m.
\\
\label{eq:stein-b}
|\partial^{\multia}\bigF(x)|
&\lesssim
n^{m+|\multia|/2} t^{-|\multia|} \Psi(x)^{|\multia|+1}
&&\text{for } |\multia| \leq m.
\\
\label{eq:stein-b-prime}
|\partial^{\multia}\bigF(x)|
&\lesssim
n^{m+|\multia|/2} t^{-|\multia|} \Psi(x)^{|\multia|+1} \delta(x)^{- |\multia|} 
&&\text{for } |\multia| \geq m+1.
\end{align}
\end{lem}

\begin{proof}
The proof is identical to the one given in \cite[Section 2.3.2]{stein1970}. To track the dependence on $n$ and $t$, we note the identity $\sum_{|\multib| = r} \frac{1}{\multib!} = \frac{n^r}{r!}$ and use \eqref{eq:delta-dist-comparable} and \Cref{lem:derivatives-of-sum-phi-star} whenever necessary.
\end{proof}

\section{Averaging}\label{sec:averaging}

In this section, we introduce a parameter $b \in \bbR^n$ to $\cE_m$ and discuss averaging over this parameter. 

\begin{defn}
Let $\tau > 0$. We say a function $F : \bbR^n \to \bbR$ is \emph{periodic with period $\tau$} if $x - x' \in \tau\bbZ^n$ implies $F(x) = F(x')$.
\end{defn}

\begin{defn}
Let $B \subset \bbR^n$ be a measurable set. Then for any function $g : B \to \bbR$, we denote the average of $g$ on $B$ by 
$
\<g(b)\>_B = \frac{1}{|B|} \int_B g(b) \, db. 
$
Furthermore, if $g : \bbR^n \to \bbR$ is periodic with period $\tau$, we set 
$
\< g(b) \> 
=
\<g(b)\>_{[0,\tau]^n}
.
$
\end{defn}

\subsection{The dependence of $\cE_m$ on the choice of origin}

The Whitney cube partition of $E^c$ depends on the choice of origin since only certain cubes in $\bbR^n$ are dyadic, and only dyadic cubes are allowed to be Whitney cubes. As a result, the extension operator $\cE_m$ depends on the choice of origin. So far, we have been working with respect to the standard origin $0=(0, \ldots, 0)$. To make this dependence explicit, we rewrite \eqref{eq:def-extm} as
\begin{align}
\label{eq:def-extm0}
\cE_{m,[0]}(f)(x)
=
\sum_k{}' P_{p_{k,[0]}}(x) \phi_{k,[0]}^*(x)
\qquad
\text{if } x \in E^c
.
\end{align}
To denote an extension operator constructed with respect to another origin $b \in \bbR^n$, we replace all the $[0]$ subscripts in \eqref{eq:def-extm0} with $[b]$.

When shifting the origin, the proper analogue of $\Psi(x)$ is 
\begin{align}
\label{eq:psi-b-x-delta}
\Psi_{[b]}(x) 
=
\sum_{c_1' \delta(x) n^{-1/2} \ \leq \ p \  \leq \  c_2'\delta(x) n^{-1/2}}
\left[
\prod_{i=1}^n \left( 
\psi\left(\frac{x_i-b_i}{p}\right) + 1
\right)
\right]
.
\end{align}
The $x-b$ arises because the dyadic cubes for the origin $b$ can by obtained by translating the dyadic cubes for the standard origin by $b$. Thus, we have the following version of \eqref{eq:stein-b} and \eqref{eq:stein-b-prime}.

\begin{lem}
\label{lem:derivatives-of-F-b}
Suppose $f = (P_y) \in \cC^{m}(E)$ with norm $\leq 1$. Let $\bigF_{[b]} = \cE_{m,[b]}(f)$. Then
\begin{align}
|\partial^{\multia}\bigF_{[b]}(x)|
&\lesssim
n^{m+|\multia|/2} t^{-|\multia|} \Psi_{[b]}(x)^{|\multia|+1}
&&\text{for } x \in E^c, |\multia| \leq m.
\\
|\partial^{\multia}\bigF_{[b]}(x)|
&\lesssim
n^{m+|\multia|/2} t^{-|\multia|} \Psi_{[b]}(x)^{|\multia|+1} \delta(x)^{- |\multia|}
&&\text{for } x \in E^c, |\multia| \geq m+1.
\end{align}
\end{lem}

\subsection{The averaged extension operator $\< \cE_m \>$}

We define the operator $\<\cE_{m}\>$ acting on $\cC^{m}(E)$ as follows:
\begin{align}
\label{eq:defn-avg-extm}
\<\cE_{m}\>(f)(x)
=
\<\cE_{m,[b]}(f)(x)\>
.
\end{align}
The definition above requires that the map $b \mapsto \cE_{m,[b]}(f)(x)$ be periodic for all $x$. We show that this is indeed the case.

\begin{lem}
\label{lem:periodic-in-b}
Fix $x \in E^c$. Then there exists $\tau > 0$ and a neighborhood $N$ of $x$ such that $\tau$ is a power of $2$ and for each $y \in N$, the function $b \mapsto \cE_{m,[b]}(f)(y)$ is periodic of period $\tau$. 
\end{lem}

\begin{proof}
Fix an $x \in E^c$ and let $\sidelengthvariable_{[0]}$ be the side length of a Whitney cube (for the standard choice of origin) that contains $x$. For any other choice of origin, any Whitney cube that contains $x$ will have side length between $\frac{1}{4}\sidelengthvariable_{[0]}$ and $4\sidelengthvariable_{[0]}$.

We claim that if we translate the origin $b$ by any element of $16\sidelengthvariable_{[0]} \bbZ^n$, the Whitney cube partition ``near $x$'' will not change. More precisely, suppose $b - b' \in 16\sidelengthvariable_{[0]} \bbZ^n$ and consider the two partitions $\{Q_{k,[b]}\}_k$ and $\{Q_{k,[b']}\}_k$. We claim that if $Q \in \{Q_{k,[b]}\}_k$ and $x \in Q^*$, then $Q \in \{Q_{k,[b']}\}_k$. This follows from how the Whitney cubes are constructed -- see the discussion on maximal cubes in\cite[Section 1.2]{stein1970}. 

Thus, $b \mapsto \cE_{m,[b]}(f)(x)$ is periodic with period $16\sidelengthvariable_{[0]}$. (Recall that $x$ is fixed and that $\sidelengthvariable_{[0]}$ depends on $x$.) Since $\sidelengthvariable_{[0]}$ is comparable with $\delta(x)$, we can extend this result to a neighborhood of $x$.
\end{proof}

\begin{lem}
\label{lem:avg-ext-maps-Cm-omega}
$\< \cE_m \>$ is an extension operator $\cC^{m}(E) \to \cC^{m}(\bbR^n)$. Furthermore, 
\begin{align}
\label{eq:derivative-under-average}
\left(\frac{\partial}{\partial x}\right)^{\multia} \< \cE_m \> (f)(x) 
= 
\< \left(\frac{\partial}{\partial x}\right)^{\multia} \cE_{m,[b]} (f)(x)  \>
.
\end{align}
\end{lem}

\begin{proof}
Let $f \in \cC^m(E)$, let $F_{[b]} = \cE_{m,[b]}(f)$, and let $F = \<\cE_m\>(f)$. Since $E$ is compact, $f \in \cC^{m,\omega}(E)$ for some modulus of continuity $\omega$. Glaeser's version of the Whitney extension theorem \cite{glaeser} implies that $F_{[b]} \in \cC^{m,\omega}(\bbR^n)$ for all $b \in \bbR^n$. Furthermore, $\sup_{b} \|F_{[b]}\|_{\cC^{m, \omega}(\bbR^n)} < \infty$. It follows that $F \in \cC^{m, \omega}(\bbR^n)$ and that \eqref{eq:derivative-under-average} holds.
\end{proof}

\subsection{Averages of $\Psi_{[b]}(x)$.} 
\label{subsec:avg-psi-b}

By averaging, we replace the sharply-peaked $\Psi_{[b]}(x)$ with something that is easier to control.

\begin{lem}
\label{lem:avg-of-Psi}
Fix $x \in E^c$ and $k \in \bbZ_{> 0}$. Then 
$
\<
\Psi_{[b]}(x)^{k}
\>
\lesssim
(1 + t2^{k+1})^n
.
$
The implied constant is independent of $n$ and $t$ (as usual) as well as of $x$, but will depend on $k$.
\end{lem}

\begin{proof}
Let $\tau$ be a power of $2$ such that $\tau \geq c_2' \delta(x) n^{-1/2}$. Then we see that $b \mapsto \Psi_{[b]}(x)$ has period $\tau$. Expanding the sum $\Psi_{[b]}(x)^k$, we have
\begin{align}
\<
\Psi_{[b]}(x)^k 
\>
&=
\sum_{p_1, \ldots, p_{k}}
\prod_{i=1}^n
\<
\prod_{r=1}^{k}
\left( 
1 + \psi\left(\frac{x_i - b_i}{p_r}\right)
\right)
\>_{[0,\tau]}
\\
\label{eq:sum-prod-prod-average}
&\leq
\sum_{p_1, \ldots, p_{k}}
\prod_{i=1}^n
\prod_{r=1}^{k}
\<
\left(
1 + \psi\left(\frac{x_i - b_i}{p_r}\right)
\right)^{k}
\>_{[0,\tau]}
^{1/k}
\end{align}
where each $p_r$ ($r = 1, \ldots k$) is a sum over powers of $2$ in the range $c_1' \delta(x) n^{-1/2} \leq p_r \leq c_2'\delta(x) n^{-1/2}$. (We apply H\"older's inequality in the second line. This is where we need $k > 0$.)

Due to our choice of $\tau$, each average in \eqref{eq:sum-prod-prod-average} is over an whole number of periods of $(1 + \psi)^k$. Note that
$
\<
\left(
1 + \psi(b)
\right)^{k}
\>_{[0,1]}
=
(1-2t) \cdot 1 + (2t) \cdot 2^{k}
\leq
1 + t2^{k+1}.
$
Also note that the number of terms in the sum over a particular $p_r$ is comparable with $\log_2(c_2'/c_1')$. Thus, 
$
\<
\Psi_{[b]}(x)^{k}
\>_B
\leq
\sum_{p_1, \ldots, p_{k}}
(1 + t2^{k+1})^n
\lesssim
(1 + t2^{k+1})^n
,
$
where the implied constant is independent of $n,x,t$ (but will depend on $k$).
\end{proof}

\subsection{Bounds on the norm of $\< \cE_m \>$}\label{sec:bounds-on-norm-of-averaged}

Up until now, $t$ has been an independent parameter. We now take $t = 1/n$. With this choice, we have ${(1+n^{-1} 2^{k+1})^n \leq e^{2^{k+1}}}$ for all $n$. (In contrast, if we keep $t$ constant as $n$ varies, then $(1 + t2^{k+1})^n$ grows exponentially in $n$.)

\begin{lem}
Let $f \in \cC^m(E)$ with norm $\leq 1$. Let $\bigF = \< \cE_m \>(f)$. Then
\begin{align}
|\partial^{\multia}\bigF(x)|
&\lesssim
n^{m+3|\multia|/2}
&&\text{for } x \in E^c, |\multia| \leq m.
\\
|\partial^{\multia}\bigF(x)|
&\lesssim
n^{m+3|\multia|/2} \delta(x)^{- |\multia|}
&&\text{for } x \in E^c, |\multia| \geq m+1.
\end{align}
\end{lem}

\begin{proof}
Let $\bigF_{[b]} = \cE_{m,[b]}(f)$. Fix $x \in E^c$ and $|\alpha| \leq m$. Using \Cref{lem:derivatives-of-F-b} and \Cref{lem:avg-ext-maps-Cm-omega} (and letting $t = 1/n$), we have
$
|\partial^{\multia}\bigF(x)| 
\leq
\<|\partial^{\multia}\bigF_{[b]}(x)|\>
\lesssim
n^{m+3|\multia|/2} \<\Psi_{[b]}(x)^{|\multia|+1}\>
\lesssim
n^{m+3|\multia|/2}e^{2^{|\alpha|+2}}
\lesssim
n^{m+3|\multia|/2}
.
$
A similar argument proves the other inequality.
\end{proof}

\begin{cor}
\label{thm:polynomial-bound-on-norms}
The norm of the map $\< \cE_m \> : \cC^{m}(E) \to \cC^{m}(\bbR^n)$ is $O(n^{5m/2})$. The implied constant depends only on $m$.
\end{cor}

\begin{rem}
It is not hard to make small improvements to the power of $n$ in \Cref{thm:polynomial-bound-on-norms}. However, we do not know the optimal power of $n$.
\end{rem}

\appendix

\section{The norm of the restriction operator}\label{sec:restriction-norm}

\begin{lem}
\label{lem:restriction-norm}
The norm of the restriction map $\cC^{m}(\bbR^n) \to \cC^{m}(E)$  given by $F \mapsto F|_E$ is $O(n^m)$. The implied constant depends only on $m$.
\end{lem}

\begin{proof}
Let $F \in \cC^{m}(\bbR^n)$ with $\| F \|_{\cC^{m}(\bbR^n)} \leq 1$. Let $(P_y)_{y \in E} = F|_E$. Then for all $x \in E$, we have
$
|\partial^\alpha P_x(x)| = |\partial^\alpha F(x)| \leq 1.
$
Also, by Taylor's theorem, for any $x,y \in E$, there exists a $z \in \bbR^n$ on the line segment joining $x$ and $y$ such that
\begin{align}
\left|
\partial^\multia (P_{x}-P_y)(x)
\right|
&\leq
\sum_{|\multib| = m - |\multia|}
\left|
\partial^{\multia+\multib}\bigF(z) - \partial^{\multia+\multib}\bigF(y)
\right|
\frac{|x-y|^{|\beta|}}{\multib!}
\\
&=
\frac{2n^{m-|\multia|}}{(m-|\multia|)!} |x-y|^{m- |\multia|}
,
\end{align}
where we use the identity 
$
\sum_{|\multib| = r} \frac{1}{\multib!} = \frac{n^r}{r!}.
$
\end{proof}

\section*{Acknowledgments}

The author would like to thank C.~Fefferman and E.~M.~Stein for helpful conversations.

%%%%%%%
% BIBLIOGRAPHY
%%%%%%%

\providecommand{\bysame}{\leavevmode\hbox to3em{\hrulefill}\thinspace}
\providecommand{\MR}{\relax\ifhmode\unskip\space\fi MR }
\providecommand{\MRhref}[2]{%
  \href{http://www.ams.org/mathscinet-getitem?mr=#1}{#2}
}
\providecommand{\href}[2]{#2}


\begin{thebibliography}{Whi34}

\bibitem[Gla58]{glaeser}
Georges Glaeser, \emph{\'{E}tude de quelques alg\`ebres tayloriennes}, J.
  Analyse Math. \textbf{6} (1958), 1--124; erratum, insert to 6 (1958), no. 2.
  \MR{0101294 (21 \#107)}

\bibitem[LG09]{legruyer}
Erwan Le~Gruyer, \emph{Minimal {L}ipschitz extensions to differentiable
  functions defined on a {H}ilbert space}, Geom. Funct. Anal. \textbf{19}
  (2009), no.~4, 1101--1118. \MR{2570317 (2011c:46163)}

\bibitem[Ste70]{stein1970}
Elias~M. Stein, \emph{Chapter {VI}: Extensions and restrictions}, Singular
  integrals and differentiability properties of functions, Princeton
  Mathematical Series, No. 30, Princeton University Press, Princeton, N.J.,
  1970, pp.~166--195. \MR{0290095 (44 \#7280)}

\bibitem[Wel73]{wells}
John~C. Wells, \emph{Differentiable functions on {B}anach spaces with
  {L}ipschitz derivatives}, Journal of Differential Geometry \textbf{8} (1973),
  no.~1, 135--152.

\bibitem[Whi34]{whitney}
Hassler Whitney, \emph{Analytic extensions of differentiable functions defined
  in closed sets}, Transactions of the American Mathematical Society
  \textbf{36} (1934), no.~1, 63--89.

\end{thebibliography}
\end{document}